\documentclass[11pt, reqno]{amsart}
\usepackage{amsmath,amsthm,amssymb,mathrsfs}
\usepackage[abbrev,non-sorted-cites]{amsrefs}
\usepackage{color,graphicx}
\usepackage{verbatim}
\usepackage{datetime}
\usepackage{hyperref}

\theoremstyle{plain}
  \newtheorem{thm}{Theorem}[section]
  \newtheorem{lem}[thm]{Lemma}
  \newtheorem{prop}[thm]{Proposition}
   
	\newtheorem*{thm*a}{Theorem A}
	\newtheorem*{thm*b}{Theorem B}
\theoremstyle{definition}
  \newtheorem{defn}[thm]{Definition}
	
    \newtheorem{example}[thm]{Example}
  \newtheorem{rmk}[thm]{Remark}
  \newtheorem*{ack*}{Acknowledgement}
  \newtheorem*{ques*}{Question}
\theoremstyle{plain}

\numberwithin{equation}{section}

\newcommand\pl{\partial}
\newcommand\bpl{\bar{\partial}}

\newcommand\oh{\frac{1}{2}}
\newcommand\dd{{\mathrm d}}

\newcommand\w{\wedge}
\newcommand\sm{\sigma}
\newcommand\dt{\delta}

\newcommand\vep{\varepsilon}
\newcommand\vph{\varphi}
\newcommand\om{\omega}
\newcommand\ta{\theta}
\newcommand\gm{\gamma}
\newcommand\kp{\kappa}
\newcommand\af{\alpha}
\newcommand\bt{\beta}

\newcommand\ld{\lambda}

\newcommand\Ta{\Theta}

\newcommand\BC{\mathbb{C}}

\newcommand\BR{\mathbb{R}}

\newcommand\BZ{\mathbb{Z}}

\newcommand\FF{\mathfrak{F}}

\newcommand\td{\tilde}

\newcommand\bfI{\mathbf{I}}

\DeclareMathOperator{\re}{Re}
\DeclareMathOperator{\im}{Im}

\begin{document}

\title{An ansatz for constructing  explicit solutions of \\ Hessian equations}

\author{Chung-Jun Tsai}
\address{Department of Mathematics, National Taiwan University, and National Center for Theoretical Sciences, Math Division, Taipei 10617, Taiwan}
\email{cjtsai@ntu.edu.tw}

\author{Mao-Pei Tsui}
\address{Department of Mathematics, National Taiwan University, and National Center for Theoretical Sciences, Math Division, Taipei 10617, Taiwan}
\email{maopei@math.ntu.edu.tw}

\author{Mu-Tao Wang}
\address{Department of Mathematics, Columbia University, New York, NY 10027, USA}
\email{mtwang@math.columbia.edu}


\thanks{C.-J.~Tsai is supported in part by the National Science and Technology Council grant 112-2628-M-002-004-MY4.  M.-P.~Tsui is supported in part by the National Science and Technology Council grant 112-2115-M-002-015-MY3 and 113-2918-I-002-004. M.-T. ~Wang is supported in part by the National Science Foundation under Grants DMS-2104212 and DMS-2404945, and by the Simons Foundation through the Simons Fellowship SFI-MPS-SFM-00006056. Part of this work was carried out when M.-T.~Wang was visiting the Institute of Mathematics, Academia Sinica.}

\begin{abstract} We introduce a (variation of quadrics) ansatz for constructing explicit, real-valued solutions to broad classes of complex Hessian equations on domains in $\mathbb{C}^{n+1}$
 and real Hessian equations on domains in $\mathbb{R}^{n+1}$.
In the complex setting, our method simultaneously addresses the deformed Hermitian--Yang--Mills/Leung--Yau--Zaslow (dHYM/LYZ) equation, the Monge--Amp\`{e}re equation, and the $J$-equation. Under this ansatz each PDE reduces to a second-order system of ordinary differential equations admitting explicit first integrals. These ODE systems integrate in closed form via abelian integrals, producing wide families of explicit solutions together with a detailed description. In particular, on $\mathbb{C}^{n+1}$,
 we construct entire dHYM/LYZ solutions of arbitrary subcritical phase, and on $\mathbb{R}^{n+1}$ we produce entire special Lagrangian solutions of arbitrary subcritical phase. Some of these solutions develop singularities on compact regions.  In the special Lagrangian case we show that, after a natural extension across the singular locus, these blow-up solutions coincide with previously known complete special Lagrangian submanifolds obtained via a different ansatz.
\end{abstract}

\maketitle


\section{Introduction}

Let $X\subset\mathbb{C}^{n+1}$ be a domain and let $u\in C^2(X)$ be a real-valued function.  We study the complex Hessian equation:
\begin{align}
    c_n\sm_{n+1}(\partial\bar{\partial}u) + c_{n-1}\sm_n(\partial\bar{\partial}u) + \cdots + c_0\sm_1( \partial\bar{\partial}u) + c_{-1} &= 0
\label{complex H_eqn} \end{align}
where $c_{-1}, c_0,\cdots, c_n$ are real constants, and $\sigma_k(\partial\bar{\partial} u), k=1,\cdots, n+1$ are the $k$-th symmetric functions of the complex Hessian of $u$. The coefficients $c_{-1}, c_0,\cdots, c_n$ of \eqref{complex H_eqn} determine two polynomials $F$ and $G$ in the variables $p_1, \cdots, p_n$:
\begin{equation}
    F(p_1, \cdots, p_n)= \sum_{k=0}^n c_k\sm_k(p) \quad\text{and}\quad G(p_1,\ldots,p_n) = \sum_{k=0}^n c_{k-1}\sm_k(p)
\end{equation} where $\sigma_k (p), k=0,\cdots n$ are the $k$-the symmetric functions of $p_i, i=1,\cdots,n$. 

By introducing a suitable ansatz, we reduce the PDE \eqref{complex H_eqn} to a system of second-order ordinary differential equations determined by $F$ and $G$. Concretely, if $p_i(s)$, $i=1,\dots,n$ are $C^2$ functions of $s$ that satisfy the ODE system
\begin{align}\label{ODE_introduction} \begin{split}
    \frac{p_i''}{2}\cdot F(p_1,\ldots,p_n) & = (p_i')^2 \frac{\pl F}{\pl p_i}(p_1,\ldots,p_n)  \quad\text{for $i=1,\ldots,n$ } ~,  \\
    r''\cdot F(p_1,\ldots,p_n) & = -G(p_1,\ldots,p_n) ~,
\end{split} \end{align}
then the function
$$
u(z_1,\dots,z_{n+1}) = 2\sum_{j=1}^n p_j(\re z_{n+1})\,(\re z_j)^2 + 4\,r(\re z_{n+1})
$$
solves the complex Hessian equation \eqref{complex H_eqn} (see Proposition \ref{general_ODE}).  Here $(z_1,\cdots, z_{n+1})$ are standard complex coordinates on $\mathbb{C}^{n+1}$ and $\re{z}_i$ denotes the real part of $z_i, i=1\cdots n+1$ .

The same ansatz—and the very same reduction to an ODE system—applies to the corresponding real Hessian equation on $\mathbb{R}^{n+1}$, i.e. \[f(x_1,\ldots, x_n, x_{n+1})=\frac{1}{2}\sum_{j=1}^n p_j(x_{n+1})\,x_j^2 + r(x_{n+1})\] satisfies the equation 
\begin{align}
    c_n\sm_{n+1}(\nabla^2 f) + c_{n-1}\sm_n(\nabla^2 f) + \cdots + c_0\sm_1(\nabla^2 f) + c_{-1} &= 0
\label{real H_eqn} \end{align} where $\sigma_k(\nabla^2 f), k=1,\cdots, n+1$ are the $k$-th symmetric functions of the real Hessian of $f$.

We then focus on Hessian equations whose coefficients $c_0,\cdots, c_n$ satisfy a \emph{recursive relation}. 
\begin{defn}
    Let $a_0$ and $a_1$ be real numbers. The complex Hessian equation \eqref{complex H_eqn}/real Hessian equation \eqref{real H_eqn} is said to be of {\emph recursive type} $(a_0, a_1)$ if the coefficients $c_0,c_1,\cdots c_n$ satisfy the recursive relation:
    \begin{align*}
         c_{k-1} = a_1 c_k  - a_0 c_{k+1} \quad \text{for } k=1, \dots, n-1
    \end{align*}
\end{defn}
In particular, the coefficients $c_0,\cdots, c_n$ are determined by $a_0, a_1, c_{n-1}$ and $c_{n}$. Many classical nonlinear PDEs lie in this class—including the deformed Hermitian--Yang--Mills/Leung--Yau--Zaslow (dHYM/LYZ) equation, the real and complex Monge--Ampère equation, the $J$-equation, and the special Lagrangian equation---so that one may recover each by choosing the appropriate recursive relation (see Proposition \ref{classification} for a complete classification of recursive-type equations).

For any such \emph{recursive-type} equation, we show that the associated second-order ODE system is \emph{completely integrable}: in particular, it admits enough first integrals to reduce the dynamics to quadratures.
\begin{thm} \label{thm_intro1}
Suppose \eqref{complex H_eqn}/\eqref{real H_eqn} is of recursive type $(a_0, a_1)$. Then its associated 2nd order ODE system \eqref{ODE_introduction} is completely integrable. In fact, denoting
\begin{align*}
    \xi_i = \frac{p_i^2 + a_1 p_i + a_0}{p'_i}, \quad i=1, \dots, n ~,
\end{align*}
then \begin{align*} \frac{F^2}{\prod_{i=1}^n p_i'} \quad and \quad
    \xi_i - \xi_1, \quad i=2, \dots, n
\end{align*}
are first integrals of the system.  Moreover, $\xi_i,i=1\cdots n$ satisfy the following ODE system:
\begin{align*}
    (\xi'_i)^2=k_1+ k_2 \prod_{j=1}^n\xi_j
\end{align*}
for explicit real constants $k_1,k_2$ depending only on $a_0, a_1, c_{n-1}$ and $c_{n}$.
\end{thm}
{When $n=3$, it turns out that the Hamiltonian system \eqref{ODE_introduction} is always completely integrable, even if its associated Hessian PDE is not of recursive type.  More precisely, when $n=3$, all ``variations of quadrics solutions" to the constant-coefficient Hessian equations \eqref{complex H_eqn}/\eqref{real H_eqn} can be completely classified and always admit three independent first integrals.}

Theorem~\ref{thm_intro1} can be used to construct non-polynomial entire solutions for both dHYM/LYZ equation on $\BC^{n+1}$ and special Lagrangian equation on $\BR^{n+1}$, when $n\geq 3$.  The equations are of recursive type $(a_0=1, a_1=0)$ and take the following form:
\begin{align*}
    \cos\ta(\sm_1 - \sm_3 + \cdots + (-1)^{k-1}\sm_{2k-1} + \cdots) - \sin\ta(1 - \sm_2 + \cdots + (-1)^{k-1}\sm_{2k-2} + \cdots) &= 0
\end{align*}
for some $\ta\in\BR$.  Note that $\ta$ and $\ta + \pi$ give equivalent equations.  On the other hand, for a $C^2$ function $u$ defined on a domain of $\BC^{n+1}$, one can consider $\Ta = \sum_{j=1}^{n+1}\arctan\ld_j$, where $\ld_j$'s are eigenvalues of $\pl\bpl u$.  This real-valued function $\Ta$ is said to be the \emph{phase of $u$}.  If $u$ solves the dHYM/LYZ equation, $\Ta$ is a constant, and $\Ta - \ta \in \BZ\pi$.  However, for a priori estimate and relevant PDE techniques, the value of $\Ta$ matters.  If $|\Ta| = (n-1)\frac{\pi}{2}$, the function is said to be of \emph{critical phase}; the range $|\Ta| > (n-1)\frac{\pi}{2}$ is called \emph{supercritical phase}; the range $|\Ta| < (n-1)\frac{\pi}{2}$ is referred as the \emph{subcritical phase}.  Known results are primarily concentrated in the critical and supercritical phases; see \cites{Jacob-Yau, CJY20, CS22, CLT24, Lin23, Lin23a} for dHYM/LYZ equation and \cites{Yuan06, WY14, Yuan20} for special Lagrangian equations. For the relation between these conditions and a priori estimates of Hessian equations, we refer to \cites{Krylov82, Evans82, CNS85}.  We apply Theorem \ref{thm_intro1} to subcritical, entire solutions to both dHYM/LYZ equation and special Lagrangian equation.

\begin{thm}[Theorem~\ref{thm_entiren_sub} and Theorem~\ref{thm_slag_entire}]
    For any integer $n\geq 3$ and any subcritical phase $\Ta$,
    \begin{itemize}
        \item there exist subcritical, non-polynomial entire solutions to the dHYM/LYZ equation on $\BC^{n+1}$ with phase $\Ta$;
        \item there exist subcritical, non-polynomial entire solutions to the special Lagrangian equation on $\BR^{n+1}$ with phase $\Ta$.
    \end{itemize}
\end{thm}

Entire solutions to the special Lagrangian equation on $\BR^3$ were previously constructed by Warren \cites{Warren16a, Warren16b} (with phase ${\pi}/{2}$) and by Li \cite{Li21} (with phase $0$).

The non-entire solutions produced by Theorem \ref{thm_intro1} develop singularities on compact regions.  In the special Lagrangian case we show in Section ~\ref{sec_sLag} that, after a natural extension across the singular locus, these blow-up solutions coincide with previously known complete special Lagrangian submanifolds obtained via a different ansatz studied by Harvey-Lawson~\cite{Harvey-Lawson}, Lawlor~\cite{Lawlor}, and Joyce~\cite{Joyce}.

 Section~\ref{sec_LYZ} is devoted to the dHYM/LYZ equation. In Section~\ref{sec_entire}, we investigate entire solutions to the dHYM/LYZ equation on $\mathbb{C}^{n+1}$. In Section~\ref{sec_sLag}, we turn to the special Lagrangian equation, demonstrating that the blow-up solutions obtained earlier can be extended to complete special Lagrangian submanifolds. In Section~\ref{sec_general}, we deal with general equations of recursive type, {and classify equations of non-recursive type when $n=3$}. The appendix contains two important calculation lemmas.

\section{The deformed Hermitian--Yang--Mills/Leung--Yau--Zaslow Equation} \label{sec_LYZ}
The Leung--Yau--Zaslow (LYZ) equation, also known as the deformed Hermitian--Yang--Mills (dHYM) equation in the literature (see Collins--Xie--Yau \cites{Collins-Xie-Yau, Jacob-Yau})  is a fully nonlinear partial differential equation. It governs a Hermitian metric on a line bundle over a K\"ahler manifold, or more generally for a real $(1,1)$-form. Suppose $(X,\omega)$ is a K\"ahler manifold and $[\alpha] \in H^{1,1}(X,\mathbb{R})$ is a $(1, 1)$ class. The case of a line bundle consists of setting $[\alpha] = c_1(L)$ where $c_1(L)$ is the first Chern class of a holomorphic line bundle $L \to X$. Suppose that the complex dimension of $X$ is $n+1$ and consider the topological constant
\[
\hat{z}([\omega],[\alpha]) = \int_X (\omega + i\alpha)^{n+1}.
\]
Notice that $\hat{z}$ depends only on the class of $\omega$ and $\alpha$. Suppose that $\hat{z} \neq 0$. Then this is a complex number
\[
\hat{z}([\omega],[\alpha]) = re^{i\theta}
\]
for some real $r>0$ and angle $\theta \in (-\pi,\pi]$ which is uniquely determined.

Fix a smooth representative differential form $\alpha$ in the class $[\alpha]$. For a smooth function $u : X \to \mathbb{R}$, the dHYM/LYZ equation for $(X,\omega)$ with respect to $[\alpha]$ is
\[
\begin{cases}
\operatorname{Im} (e^{-i\theta} (\omega + i(\alpha+\frac{i}{2}\partial\bar{\partial}u))^{n+1}) = 0 \\
\operatorname{Re} (e^{-i\theta} (\omega + i(\alpha+\frac{i}{2}\partial\bar{\partial}u))^{n+1}) > 0.
\end{cases}
\]


Take $X$ to be a domain of $\BC^{n+1}$, $\alpha=0$,  and $\omega=\frac{i}{2}\sum_{j=1}^{n+1} dz_j\wedge d\bar{z}_k$ for the standard complex coordinates $z_1,\ldots, z_{n+1}$ of $\BC^{n+1}$, the LYZ equation for $u: X\rightarrow \BR$ becomes 
\begin{align} \label{LYZ} &\im\Bigl( e^{-i\ta}\det\Bigl(\bfI_{n+1} + i\Bigl[\frac{\partial^2 u}{\partial{z_j}\partial{\bar{z}_k}}\Bigr]_{1\leq j, k\leq n+1} \Bigr)\Bigr) =0 ~,\\
&\re\Bigl( e^{-i\ta}\det\Bigl(\bfI_{n+1} + i\Bigl[\frac{\partial^2 u}{\partial{z_j}\partial{\bar{z}_k}}\Bigr]_{1\leq j, k\leq n+1} \Bigr)\Bigr) > 0 ~.
\end{align}
Recall that the sum of the arctangent of the eigenvalues of $\bigl[\frac{\partial^2 u}{\partial{z_j}\partial{\bar{z}_k}}\bigr]_{1\leq j, k\leq n+1}$ is called the phase of $u$, and belongs to $(-\frac{n+1}{2}\pi,\frac{n+1}{2}\pi)$.  If $u$ satisfies \eqref{LYZ}, its phase is a constant, and is equal to $\ta$ modulo $\pi\BZ$.

Our ansatz assumes the potential function $u$ is of the form:
\begin{align}\label{ansatz_u}   
    u(z_1,\ldots,z_n,z_{n+1}) &= \sum_{j=1}^n 2p_j(s)(x_j)^2 + 4r(s)
\end{align}
where $p_j(s)$ and $r(s)$ are real-valued functions in $s = \re z_{n+1}$, and $x_j=\re z_j, j=1,\ldots, n$.

We compute
\begin{align*}
    \frac{\pl^2 u}{\pl z_j\pl\bar{z}_k} &= p_j(s) \dt_{jk}  \quad\text{for } j, k=1,\ldots, n ~, \\
    \frac{\pl u}{\pl z_j\pl\bar{z}_{n+1}} &= p'_j(s) x_j  \quad\text{for } j=1,\ldots,n ~, \\
    \frac{\pl^2 u}{\pl z_{n+1}\pl\bar{z}_{n+1}} &= \sum_{j=1}^n \oh p''_j(s)(x_j)^2 + r''(s) ~.
\end{align*}
It follows that the coefficient matrix of $\pl\bpl u$ is
\begin{align} \label{D2u}
     & \begin{bmatrix}
    p_1 & 0 & \cdots & 0 & x_1 p'_1 \\  0 & p_2 & \cdots & 0 & x_2 p'_2 \\  \vdots & \vdots & \ddots & \vdots & \vdots \\
    0 & 0 & \cdots & p_{n} & x_n p'_{n} \\  x_1 p'_1 & x_2 p'_2 & \cdots & x_n p'_{n} & \left( \sum_{j=1}^n \oh(x_j)^2 p''_j(s) + r''(s) \right)
    \end{bmatrix} ~.
\end{align}
With Lemma \ref{lem_det} in the appendix, we compute
\begin{equation}
    \begin{split} \label{I+iD2u}
     &\det\Bigl(\bfI_{n+1} + i\Bigl[\frac{\partial^2 u}{\partial{z_j}\partial{\bar{z}_k}}\Bigr]_{1\leq j, k\leq n+1} \Bigr)\\
     &= (1+ip_1)\cdots(1+ip_n)\biggl(1 + i\biggl(\sum_{j=1}^n \oh(x_j)^2 p''_j(s) + r''(s)\biggr)\biggr) \\
    &\quad + \sum_{j=1}^n(x_j)^2(p_j'(s))^2(1+ip_1)\cdots\widehat{(1+ip_j)}\cdots(1+ip_n) \\
    &= \FF\cdot \biggl(1 + i\biggl(\sum_{j=1}^n \oh(x_j)^2 p''_j(s) + r''(s)\biggr)\biggr) - i \sum_{j=1}^n(x_j)^2(p_j'(s))^2 \frac{\pl\FF}{\pl p_j} ~
 \end{split}
\end{equation}
where \begin{equation}\label{F} \FF = (1+ip_1)\cdots(1+ip_n) ~. \end{equation}

Denoting \begin{equation}\label{F_theta} F_\ta = \re(e^{-i\ta} \FF),\end{equation} we obtain:

\begin{defn} \label{theta_ODE} For any $\theta\in (-\pi, \pi]$,  $p_1(s), \cdots, p_n(s)$ and $r(s)$ are said to satisfy the $\theta$-angle ODE system on the interval $I\subset \BR$ if 
    \begin{align}
    F_\ta(p_1,\ldots,p_n)\,\frac{p''_j}{2} &= \frac{\pl F_\ta}{\pl p_j}\,(p'_j)^2 \quad\text{ for } j\in\{1,\ldots,n\} \quad\text{and}  \label{ODE_p} \\
    F_\ta(p_1,\ldots,p_n)\,r'' &= -F_{\ta+\frac{\pi}{2}}(p_1,\ldots,p_n) ~,  \label{ODE_r}
\end{align}
for $F_\ta = \re(e^{-i\ta} (1+ip_1)\cdots(1+ip_n))$  and any $s\in I$ . As we will see in Lemma \ref{lem_Ft}, $\ta$ and $\ta+\pi$ indeed correspond to equivalent ODE systems.
\end{defn}

\begin{prop} \label{prop_basic}
Suppose $p_1(s), \cdots, p_n(s)$ and $r(s)$ satisfy the ODE system \eqref{ODE_p} \eqref{ODE_r} with $p_j'(s)\not=0$ and $F_\ta(p_1,\ldots, p_n)>0$. The function $u$ formed by \eqref{ansatz_u} satisfies the dHYM/LYZ equation \eqref{LYZ} on the domain $X$.
\end{prop}

It follows that
\begin{align}\label{ODE_p2}
   (\frac{1}{p_j'}) ' &= - 2\pl_{p_j} (\log F_\ta) 
\end{align} and
\begin{align*}
    r'' &= - \frac{F_{\ta + \frac{\pi}{2}}}{F_\ta} ~.
\end{align*}

\begin{lem} \label{lem_Ft}
    For any $\ta\in\BR$, the polynomial $F_{\ta}$ defined by \eqref{F_theta} has the following properties.
    \begin{enumerate}
        \item $F_{\ta} = -F_{\ta+\pi}$.
        \item $(F_\ta)^2 + (F_{\ta+\frac{\pi}{2}})^2 = \prod_{j=1}^n(1+p_j^{\,2})$.
        \item For any $j\in\{1,\ldots,n\}$,
        \begin{align*}
            p_j\,F_\ta - (1+p_j^{\,2})\,\frac{\pl F_\ta}{\pl p_j} &= F_{\ta+\frac{\pi}{2}} ~.
        \end{align*}
    \end{enumerate}
\end{lem}
\begin{proof}
Properties (i) and (ii) follow directly from the definition.  We compute
\begin{align*}
    p_j \FF - (1+p_j^{\,2})\frac{\pl \FF}{\pl p_j} &= \left(p_j - i(1-i p_j)\right) \FF = -i\FF ~,
\end{align*}
and property (iii) follows.
\end{proof}

\subsection{First integrals of the ODE system}

In this section, we show that the ODE system \eqref{ODE_p} admits $n$ first integrals and the system is integrable. 

Suppose that $p'_j$ and $F_\ta(p_1,\ldots,p_n)$ are nonzero.  By \eqref{ODE_p},
\begin{align*}
    \frac{\dd}{\dd s} \biggl(\frac{\prod_{j=1}^n p'_j}{(F_\ta)^2}\biggr) &= \frac{\prod_{k=1}^n p'_k}{(F_\ta)^2} \sum_{j=1}^n \biggl( \frac{p''_j}{p'_j} F_\ta - 2\frac{\pl F_\ta}{\pl p_j}p'_j \biggr) = 0 ~.
\end{align*}
Thus, there exists a constant $\kp\neq 0$ such that
\begin{align}
    \frac{\prod_{j=1}^n p'_j}{(F_\ta)^2} &= {\kp} ~.  \label{conserve0}
\end{align}

For $j\in\{1,\ldots,n\}$, let
\begin{align}
    \xi_j &= \frac{1+p_j^{\,2}}{p'_j} ~.  \label{xi_def}
\end{align}
According to \eqref{ODE_p} and Lemma \ref{lem_Ft} (iii),
\begin{align}
    \xi'_j &= 2p_j - (1+p_j^{\,2})\frac{p''_j}{(p'_j)^2}  \notag \\
    &= \frac{2}{F_\ta} \biggl( p_j\,F_\ta - (1+p_j^{\,2})\frac{\pl F_\ta}{\pl p_j} \biggr) = \frac{2F_{\ta + \frac{\pi}{2}}}{F_{\ta}}  \label{xi_der}
\end{align}
for any $j\in\{1,\ldots,n\}$.  Hence, there exist constants $\kp_2,\ldots,\kp_n$ such that
\begin{align}
    \xi_1 - \xi_j = \kp_j  \label{xi_rel}
\end{align}
for $j\in\{2,\ldots,n\}$.  It is convenient to set $\kp_1$ to be $0$.

By \eqref{xi_der}, Lemma \ref{lem_Ft} (ii), \eqref{xi_def} and \eqref{conserve0},
\begin{align}
    \frac{(\xi'_j)^2}{4} &= \biggl(\frac{F_{\ta+\frac{\pi}{2}}}{F_\ta}\biggr)^2 = \frac{\prod_{k=1}^n(1+p_k^{\,2})}{(F_\ta)^2} - 1 = \frac{\prod_{k=1}^n(p'_i\,\xi_k)}{(F_\ta)^2} - 1 \notag \\
    &= \kp\prod_{k=1}^n\xi_k - 1  \label{xi_eqn1}
\end{align}
for $j\in\{1,\ldots,n\}$.  Consider $j=1$, and apply \eqref{xi_rel}; one finds that
\begin{align} \label{xi_eqn2}
    \xi'_1 &= \pm2\sqrt{\kp\prod_{j=1}^n(\xi_1-\kp_j) - 1} ~.
\end{align}

\subsection{Limits of $s$ when $n\geq 3$}

It follows from the definition \eqref{xi_def} of $\xi_j$ and $p_k'\neq 0$ that $\xi_j$ does not change sign.
We would like to argue that $p_k(s)$ cannot be defined for all $s\in\BR$ when $n\geq 3$.

At first, suppose that $\xi_i\geq 0$ and is bounded from above, $\xi_i \leq C$ for some $C > 0$.  By \eqref{xi_def}, $\arctan p_k = \int\frac{1}{\xi_k}\dd s$, and the lower bound of $\xi_k$ implies that $p_k$ must blow up for finite $s$.

On the other hand, suppose that $\xi_1$ is unbounded from above.  It follows from \eqref{xi_eqn2} that
\begin{align*}
    \int \dd s &= \frac{\pm1}{2} \int \frac{\dd \xi_1}{\sqrt{\kp\prod_{j=1}^n(\xi_1-\kp_j) - 1}} ~.
\end{align*}
If one considers the improper integral of the right hand side (to $\xi_1 = \infty$), it diverges only when $n = 2$.  Therefore, $p_k(s)$ cannot be defined for all $s\in\BR$ if $n\geq 3$.

\subsection{The isotropic Case}

In this subsection, we consider the \emph{isotropic} case of Proposition \ref{prop_basic}.  That is, $p_1 = \cdots = p_n$.  Abbreviate them as $p$, and let $\vph = \arctan p$.  Equations \eqref{ODE_p} and \eqref{ODE_r} read as follows.
\begin{align}
    (1+p^2)^\oh\,\re(e^{-i\ta}e^{in\vph})\,\frac{p''}{2} &= -\im(e^{-i\ta}e^{i(n-1)\vph})\,(p')^2 ~, \label{iso_p} \\
    \re(e^{-i\ta}e^{in\vph})\,r'' &= -\im(e^{-i\ta}e^{in\vph}) ~, \label{iso_r}
\end{align}
and we assume that $p' \neq 0$ and $\re(e^{-i\ta}e^{in\vph}) \neq 0$.  Note that \eqref{iso_r} implies that
\begin{align} \label{angle_r2}
    r'' = \tan(\ta - n\vph)  \quad\Leftrightarrow\quad  \arctan r'' = \ta - n\vph + k\pi
\end{align}
where $k$ is the unique integer such that $|\ta - n\vph + k\pi| < \frac{\pi}{2}$.

For a solution to \eqref{LYZ}, its phase be evaluated at $x_1 = \cdots = x_n = 0$.  By \eqref{angle_r2}, the phase is
\begin{align*}
    n \arctan p + \arctan r'' &= n\vph + \ta - n\vph + k\pi = \ta + k\pi ~.
\end{align*}

The above discussion gives the following:
\begin{align*}
    (p')^n &= \kp\bigl(\re(e^{-i\ta}(1+ip)^n)\bigr)^2 ~,
\end{align*}
or equivalently, in terms of $\vph$,
\begin{align*}
    (\vph')^n &= \kp \bigl( \cos(n\vph - \ta) \bigr)^2 ~.
\end{align*}
One infers that
\begin{align}
    (n\vph - \ta)' &= \kp' \bigl( \cos(n\vph - \ta) \bigr)^{\frac{2}{n}}
\end{align}
for some constant $\kp'$.
By analyzing the linearization at where $n\vph - \ta - \frac{\pi}{2} \in \BZ \pi$, it is not hard to find that for $n\geq 3$, $n\vph - \ta$ cannot be defined for all $s$.

\begin{prop}
    When $n\geq 3$ and $p_1(s) = \cdots = p_n(s)$, there is no non-constant entire solution to \eqref{ODE_p} and \eqref{ODE_r}.
\end{prop}

\section{Entire solutions of dHYM/LYZ} \label{sec_entire}

\subsection{On $\mathbb{C}^3$}

When $n=2$, i.e., on $\BC^3$, \eqref{xi_eqn2} can be solved explicitly and we obtain explicit solutions to the dHYM/LYZ equation.  In particular, when the constant $\kp$ \eqref{conserve0} is positive, the solution is defined on the whole space.

\subsubsection{When $\kp$ is positive} \label{sec_entn2}

When $\kp > 0$, the polynomial $\kp\,\xi_1^2 - \kp\,\kp_2\,\xi_1 - 1$ must have one positive root and one negative root.  Denote its roots by $\af^2$ and $-\bt^2$ for $\af,\bt > 0$.  It follows that $\kp = (\af \bt)^{-2}$, $\kp_2 = \af^2 - \bt^2$, and \eqref{xi_eqn2} becomes
\begin{align*}
    \pm1 & = \frac{\af\bt}{2} \frac{\xi_1'}{\sqrt{(\xi_1 - \af^2)(\xi_1 + \bt^2)}} ~.
\end{align*}
We now assume that $\xi_1 > \af^2$, and the case where $\xi_1 < - \bt^2$ is similar.
By integrating both sides and translating $s$,
\begin{align*}
    \tanh(\frac{s}{\af\bt}) = \frac{\sqrt{\xi_1 - \af^2}}{\sqrt{\xi_1 + \bt^2}} 
    \qquad\Rightarrow\quad \xi_1 = \af^2\cosh^2(\frac{s}{\af\bt}) + \bt^2\sinh^2(\frac{s}{\af\bt}) ~.
\end{align*}
Together with \eqref{xi_rel},
\begin{align*}
    \xi_2 &= \bt^2\cosh^2(\frac{s}{\af\bt}) + \af^2\sinh^2(\frac{s}{\af\bt}) ~.
\end{align*}

With $\xi_1$, $p_1$ can be found by \eqref{xi_def}:
\begin{align}
    (\arctan p_1)' &= \frac{1}{\xi_1} = \frac{1}{\af^2\cosh^2(\frac{s}{\af\bt}) + \bt^2\sinh^2(\frac{s}{\af\bt})}  \notag \\
    \Rightarrow\quad \arctan p_1 &= \arctan\left(\frac{\bt}{\af}\tanh(\frac{s}{\af\bt})\right) + \psi_1  \notag \\
    \Rightarrow\quad p_1 &= \frac{\af\sin\psi_1\cosh(\frac{s}{\af\bt}) + \bt\cos\psi_1\sinh(\frac{s}{\af\bt})}{\af\cos\psi_1\cosh(\frac{s}{\af\bt}) - \bt\sin\psi_1\sinh(\frac{s}{\af\bt})} ~,  \label{n2p1}
\end{align}
for some $\psi_1\in\BR$.  Similarly,
\begin{align}
    p_2 &= \frac{\bt\sin\psi_2\cosh(\frac{s}{\af\bt}) + \af\cos\psi_2\sinh(\frac{s}{\af\bt})}{\bt\cos\psi_2\cosh(\frac{s}{\af\bt}) - \af\sin\psi_2\sinh(\frac{s}{\af\bt})}  \label{n2p2}
\end{align}
for some $\psi_2\in\BR$.

Using \eqref{conserve0} and Proposition \ref{prop_basic}, one finds that
\begin{align} \label{Ftan2}
    F_{\ta} &= \frac{\af\bt}{({\af\cos\psi_1\cosh(\frac{s}{\af\bt}) - \bt\sin\psi_1\sinh(\frac{s}{\af\bt})}) ({\bt\cos\psi_2\cosh(\frac{s}{\af\bt}) - \af\sin\psi_2\sinh(\frac{s}{\af\bt})})} ~.
\end{align}
We compute
\begin{align}
    \FF &= (1 + ip_1)(1 + ip_2) \notag \\
    &= \frac{F_{\ta}}{\af\bt} e^{i(\psi_1 + \psi_2)} \left( \af\bt + \frac{i}{2}(\af^2+\bt^2)\sinh(\frac{2s}{\af\bt}) \right) ~. \label{FF2}
\end{align}
It follows that
\begin{align}
    e^{i\ta} &= e^{i(\psi_1 + \psi_2)} ~, \\
    F_{\ta+\frac{\pi}{2}} &= \frac{\af^2 + \bt^2}{2\af\bt} \sinh(\frac{2s}{\af\bt})\,F_\ta ~. \notag
\end{align}
By \eqref{ODE_r},
\begin{align}
    r'' = - \frac{\af^2 + \bt^2}{2\af\bt} \sinh(\frac{2s}{\af\bt}) \quad\Rightarrow\quad
    r &= - \frac{\af\bt(\af^2 + \bt^2)}{8} \sinh(\frac{2s}{\af\bt}) ~,  \label{n2r}
\end{align}
up to adding an affine function in $s$.  With the explicit formulae \eqref{n2p1}, \eqref{n2p2} and \eqref{n2r}, the phase of \eqref{D2u} is a constant, and is equal to
\begin{align}
    \psi_1 + \psi_2
\end{align}

In order for these expressions to be defined for all $s$, the denominators have to be nonzero for all $s$.  It means that
\begin{align*}
    \frac{\af}{\bt} \geq \tan\psi_1\tanh(\frac{s}{\af\bt})  \quad\text{and}\quad
    \frac{\bt}{\af} \geq \tan\psi_2\tanh(\frac{s}{\af\bt})
\end{align*}
for all $s$, and thus
\begin{align}
    \frac{\af}{\bt} \geq |\tan\psi_1|\quad\text{and}\quad
    \frac{\bt}{\af} \geq |\tan\psi_2|  \label{n2angle1}
\end{align}
It follows from $\arctan \frac{\af}{\bt} + \arctan \frac{\af}{\bt} = \frac{\pi}{2}$ that the phase, $\psi_1 + \psi_2$, belongs to $[-\frac{\pi}{2},\frac{\pi}{2}]$.  We summarize the discussion in the following proposition.

\begin{prop} \label{prop_entire3}
    For any positive $\af,\bt$, and $\psi_1,\psi_2\in[-\frac{\pi}{2},\frac{\pi}{2}]$ satisfying \eqref{n2angle1}, the potential function $2p_1(s)(x_1)^2 + 2p_2(s)(x_2)^2 + 4r(s)$ is an entire solution to the LYZ equation on $\BC^3$, with phase $\psi_1+\psi_2$.  Here, $s = \re z_3$, and the functions are given by \eqref{n2p1}, \eqref{n2p2} and \eqref{n2r}.  In particular, the LYZ equation with non-supercritical phase on $\BC^3$ admits entire solutions.
\end{prop}

\subsubsection{When $\kp$ is negative}

Suppose that $\kp < 0$.  From \eqref{xi_eqn1}, the polynomial $\kp\, \xi_1^2 - \kp\,\kp_2\,\xi_1 - 1$ must be positive somewhere.  Therefore, the quadratic polynomial admits two real roots of the same sign.  Assume that the roots are $\af^2 > \bt^2 > 0$, with $\af,\bt>0$.  We leave it for the readers to verify that the case of negative roots corresponds to switching the roles of $\xi_1$ and $\xi_2$ in the following discussion.

In this case, $\kp = -(\af\bt)^{-2}$ and $\kp_2 = \af^2 + \bt^2$, and \eqref{xi_eqn2} becomes
\begin{align*}
    \pm1 &= \frac{\af\bt}{2} \frac{\xi_1'}{\sqrt{(\af^2 - \xi_1) (\xi_1 - \bt^2)}} ~.
\end{align*}
With a similar computation,
\begin{align*}
    \xi_1 = \af^2\cos^2(\frac{s}{\af\bt}) + \bt^2\sin^2(\frac{s}{\af\bt})  \quad\text{and}\quad
    \xi_2 = - \bt^2\cos^2(\frac{s}{\af\bt}) - \af^2\sin^2(\frac{s}{\af\bt}) ~.
\end{align*}
By integration their reciprocals,
\begin{align*}
    p_1 &= \frac{\af\sin\psi_1\cos(\frac{s}{\af\bt}) + \bt\cos\psi_1\sin(\frac{s}{\af\bt})}{\af\cos\psi_1\cos(\frac{s}{\af\bt}) - \bt\sin\psi_1\sin(\frac{s}{\af\bt})} ~, \\
    p_2 &= \frac{\bt\sin\psi_2\cos(\frac{s}{\af\bt}) - \af\cos\psi_2\sin(\frac{s}{\af\bt})}{\bt\cos\psi_2\cos(\frac{s}{\af\bt}) + \af\sin\psi_2\sin(\frac{s}{\af\bt})}
\end{align*}
for some $\psi_1,\psi_2\in\BR$.  The denominators of $p_1$ and $p_2$ cannot be nonzero for all $s$, and the solution cannot be extended to an entire solution in this case.

\subsection{On $\BC^{n+1}$ with ${n+1\geq 4}$}

In \cite{Li21}*{Theorem 2}, Li constructed non-quadratic solutions to the special Lagrangian equation \ref{SPL} with $\ta = 0$ by using the ansatz $f = \oh p(x_3) x_1^2 + q(x_3) x_2 + r(x_3)$.

It suggests that we may obtain entire solutions in higher dimensions by modifying the solutions given by Proposition \ref{prop_entire3} based this type of ansatz.  Specifically, when $n\geq 3$, consider
\begin{align} \label{ansatz_2n}
    \td{u}(z_1,z_2,z_3,\ldots,z_n,z_{n+1}) &= 2p_1(s) (x_1)^2 + 2p_2(s) (x_2)^2 + 4\sum_{j=3}^nq_j(s) x_j + 4r(s)
\end{align}
where $s = \re z_{n+1}$ and $x_j = \re z_j$ for $j=1,\ldots,n$.  After a direct computation, the coefficient matrix of $\pl\bpl\td{u}$ is
\begin{align} \label{D2tdu}
     & \begin{bmatrix}
    p_1 & 0 & 0 & \cdots & 0 & x_1 p'_1 \\  0 & p_2 & 0 & \cdots & 0 & x_2 p'_2 \\
    0 & 0 & 0 & \cdots & 0 & q_3' \\ \vdots & \vdots & \vdots & \ddots & \vdots & \vdots \\
    0 & 0 & 0 & \cdots & 0 & q_n' \\  x_1 p'_1 & x_2 p'_2 & q_3' & \cdots & q'_{n} & \frac{1}{4}\frac{\dd^2}{\dd s^2}\td{u}
    \end{bmatrix} ~.
\end{align}
It follows that
\begin{align*}
     &\quad \det\biggl(\bfI_{n+1} + i\biggl[\frac{\partial^2 u}{\partial{z_j}\partial{\bar{z}_k}}\biggr]_{1\leq j, k\leq n+1} \biggr)\\
     &= (1+i p_1)(1+i p_2) \biggl[ 1 + (x_1)^2\frac{(p_1')^2}{1+ip_1} + (x_2)^2\frac{(p_2')^2}{1+ip_2} + \sum_{j=3}^n(q_j')^2 \\
     &\qquad\qquad\qquad\qquad + i \Bigl( \oh p''_1(s) (x_1)^2 + \oh p''_2(s) (x_2)^2 + \sum_{j=3}^nq''_j(s) x_j + r''(s) \Bigr) \biggr] ~.
\end{align*}

Let $\FF = (1+ip_1)(1+ip_2)$, and $F_\ta = \re(e^{-i\ta}\FF)$ as before.  The LYZ equation \eqref{LYZ} becomes the following system:
\begin{align}
    F_\ta(p_1,p_2) \frac{p''_j}{2} &= \frac{\pl F_\ta}{\pl p_j} (p'_j)^2 \qquad\qquad\text{for } j=1,2 ~,  \label{ent01} \\
    F_\ta(p_1,p_2)\, q_k'' &= 0 \qquad\qquad\qquad\text{for } k=3,\ldots,n ~,  \label{ent02} \\
    F_\ta(p_1,p_2)\, r'' &= -F_{\ta+\frac{\pi}{2}}(p_1,p_2)\bigl( 1+\sum_{k=3}^n(q'_k)^2 \bigr) ~.  \label{ent03}
\end{align}
The equation \eqref{ent01} is analyzed in Section~\ref{sec_entn2}.  From \eqref{ent02}, one infers that $q_k(s) = \gm_ks + \tau_k$ for some constants $\gm_k,\tau_k$.  Note that $\tau_k$'s do not show up in $\pl\bpl\td{u}$.  By comparing \eqref{ent03} with \eqref{n2r}, it is not hard to find that
\begin{align*}
    r &= - \frac{\af\bt(\af^2 + \bt^2)}{8} (1 + \sum_{k=3}^n(\gm_k)^2) \sinh(\frac{2s}{\af\bt}) ~.
\end{align*}

We summarize the discussion in the following proposition.
\begin{prop} \label{thm_entiren}
    Suppose that $n\geq 3$.  For any $\af,\bt >0$, $\psi_1,\psi_2\in[-\frac{\pi}{2},\frac{\pi}{2}]$ satisfying \eqref{n2angle1} and $\gm_3,\ldots,\gm_n\in\BR$, the function $\td{u}$ defined by
    \begin{align*}
        &2p_1(s)(x_1)^2 + 2p_2(s)(x_2)^2 + 4\sum_{k=3}^n \gm_k\,s\,x_k - 4\frac{\af\bt(\af^2 + \bt^2)}{8} \Bigl(1 + \sum_{k=3}^n(\gm_k)^2\Bigr) \sinh(\frac{2s}{\af\bt})
    \end{align*}
    is an entire solution to the LYZ equation \eqref{LYZ} on $\BC^{n+1}$ with phase $\ta = \psi_1+\psi_2$.  Here, $x_j = \re z_j$ for $j=1,\ldots,n$, $s = \re z_{n+1}$, and \(p_1(s), p_2(s)\) are defined by \eqref{n2p1}, \eqref{n2p2}.

    In other words, the LYZ equation admits non-polynomial entire solutions on $\BC^{n+1}$ with any phase within $[-\frac{\pi}{2},\frac{\pi}{2}]$.
\end{prop}

    More generally, for
    \begin{align} \label{general_potential}
        u = \sum_{j=1}^n (2p_j(s)(x_j)^2 + 4 q_j(s)x_j) + 4r(s) ~,    
    \end{align}
    the dHYM/LYZ equation \eqref{LYZ} becomes
    \begin{align}
        F_\ta(p_1,\ldots,p_n)\,\frac{p''_j}{2} &= \frac{\pl F_\ta}{\pl p_j}\,(p'_j)^2 \qquad\text{for } j = 1,\ldots,n ~, \label{with_linear_01} \\
        F_\ta(p_1,\ldots,p_n)\,\frac{q''_j}{2} &= \frac{\pl F_\ta}{\pl p_j}\,p'_j\,q'_j \qquad\text{for } j = 1,\ldots,n ~, \label{with_linear_02} \\
        F_\ta(p_1,\ldots,p_n)\,r'' &= \sum_{j=1}^n\frac{\pl F_\ta}{\pl p_j}\,(q'_j)^2 - F_{\ta+\frac{\pi}{2}}(p_1,\ldots,p_n) ~. \label{with_linear_03}
    \end{align}
By using Proposition \ref{thm_entiren}, it is not hard to construct entire solutions to the LYZ equation \eqref{LYZ} on $\BC^{n+1}$ of any subcritical phase.

\begin{thm} \label{thm_entiren_sub}
    For any $n\geq 3$, the LYZ equation \eqref{LYZ} on $\BC^{n+1}$ admits non-polynomial entire solutions with any subcritical phase.
\end{thm}

\begin{proof}
Let $\af,\bt >0$, $\psi_1,\psi_2\in[-\frac{\pi}{2},\frac{\pi}{2}]$ satisfying \eqref{n2angle1}, $\psi_3,\ldots,\psi_n\in(-\frac{\pi}{2},\frac{\pi}{2})$ and $\gm_3,\ldots,\gm_n\in\BR$.
Set the coefficient functions and the angle as follows:
\begin{itemize}
    \item Let \(p_1(s)\) and \(p_2(s)\) be defined by \eqref{n2p1} and \eqref{n2p2}, \(q_1(s) = q_2(s) = 0\).
    \item For \(k=3,\ldots,n\), let \(p_k(s) = \tan(\psi_k)\) and \(q_k(s) = \sec(\psi_k)\gm_k s\).
    \item Let \(r(s)\) be
    \begin{align} \label{entire_r}
        - \frac{\af\bt(\af^2 + \bt^2)}{8} (1 + \sum_{k=3}^n(\gm_k)^2) \sinh(\frac{2s}{\af\bt}) + \frac{1}{2}\sum_{k=3}^n\tan(\psi_k)(\gm_k s)^2 ~.
    \end{align}
    \item The angle \(\ta = \sum_{j=1}^n\psi_j\).
\end{itemize}

Denote \((1+ip_1)(1+ip_2)\) by \(\FF_0(p_1,p_2)\).  It follows that
\begin{align*}
    F_\ta(p_1,p_2,\tan(\psi_3),\ldots,\tan(\psi_n) &= \prod_{k=3}^n\sec(\psi_k) \re(e^{-i(\psi_1+\psi_2)} \FF_0(p_1,p_2)) ~.
\end{align*}
Due to Proposition \ref{thm_entiren}, \eqref{with_linear_01} and \eqref{with_linear_02} are satisfied for \(j=1,2\).  For \(j=3,\ldots,n\), both hand sides of \eqref{with_linear_01} and \eqref{with_linear_02} are zero.

It remains to verify \eqref{with_linear_03}.  Let \(\hat{F}(s)\) be defined by the right-hand side of \eqref{Ftan2}. By \eqref{FF2},
\begin{align*}
    \FF_0(p_1(s),p_2(s)) &= e^{i(\psi_1+\psi_2)} \hat{F}(s)\cdot\bigl( 1+i\frac{\af^2+\bt^2}{2\af\bt}\sinh(\frac{2s}{\af\bt}) \bigr) ~.
\end{align*}
A straightforward computation shows that
\begin{align*}
    F_{\ta}(p_1(s),\ldots,p_n(s)) &= \bigl(\prod_{k=3}^n\sec(\psi_k)\bigr)\hat{F}(s) ~, \\
    F_{\ta+\frac{\pi}{2}}(p_1(s),\ldots,p_n(s)) &= \bigl(\prod_{k=3}^n\sec(\psi_k)\bigr)\hat{F}(s) \frac{\af^2+\bt^2}{2\af\bt}\sinh(\frac{2s}{\af\bt}) ~, \\
    \frac{\pl F_\ta}{\pl p_j}(p_1(s),\ldots,p_n(s)) &= \bigl(\prod_{k=3}^n\sec(\psi_k)\bigr)\hat{F}(s) \cos(\psi_j) \Bigl( \sin(\psi_3) - \cos(\psi_3)\frac{\af^2+\bt^2}{2\af\bt}\sinh(\frac{2s}{\af\bt}) \Bigr)
\end{align*}
for \(j=3,\ldots,n\).  Therefore \eqref{entire_r} satisfies \eqref{with_linear_03}.  By examining the eigenvalues of \(\pl\bpl u\) at \(0=x_1=\cdots=x_n=s\), the phase \(\Ta = \sum_{j=1}^n\psi_j\).  It finishes the proof of this theorem.
\end{proof}

\begin{example}
    For any \(\Ta\in(-(n-1)\frac{\pi}{2},(n-1)\frac{\pi}{2})\), we set \(\psi = \Ta/n\), \(\af = \bt = 1\), and \(\gm_k = 0\). Define 
    \[p_\psi(s)=\frac{\sin\psi\cosh s + \cos\psi\sinh s}{\cos\psi\cosh s - \sin\psi\sinh s}.\] Then an entire solution $u$ of the LYZ equation with phase $\Theta$, defined on $\mathbb{C}^{n+1}$, is given by ,
    \begin{equation}\begin{split} u &= 2 p_\psi(s) \bigl((x_1)^2+(x_2)^2\bigr) + 2\tan\psi\sum_{k=3}^n(x_k)^2 - \sinh(2s)~.
     \end{split}\end{equation}
    It is of the form \eqref{ansatz_u} with $p_1(s)=p_2(s)=p_\psi(s) $ and $p_j(s)=\tan\psi, j=3, \cdots n$.  We compute
    \begin{align*}
        F_\Ta &= (\sec\psi)^{n-2}\frac{1}{(\cos\psi\cosh s - \sin\psi\sinh s)^2} ~,\\
        F_{\Ta+\frac{\pi}{2}} &= (\sec\psi)^{n-2}\frac{\sinh(2s)}{(\cos\psi\cosh s - \sin\psi\sinh s)^2} ~, \\
        \frac{\pl F_\Ta}{\pl p_j} &= (\sec\psi)^{n-2}\frac{\sin\psi\cosh s - \cos\psi\sinh s}{\cos\psi\cosh s - \sin\psi\sinh s} \text{     for \(j=1, 2\)}~.
    \end{align*}
    It is a direct computation to verify that \eqref{ODE_p} and \eqref{ODE_r} are satisfied.
    Note that the solution $u$ is defined on $\mathbb{C}^{n+1}$, and that \(\Ta\in(-(n-1)\frac{\pi}{2},(n-1)\frac{\pi}{2})\) covers the whole subcritical range. 
\end{example}

\section{Solutions to the special Lagrangian equation} \label{sec_sLag}

According to Leung--Yau--Zaslow in \cite{Leung-Yau-Zaslow}, a dHYM connection is mirror to a special Lagrangian sections via the Fourier-Mukai transform under the setting of semi-flat Calabi--Yau metrics.  If one works out the transformation with respect to the standard metric on $\BC^{n+1}$, each of the solutions of the dHYM/LYZ equation we obtained in previous sections corresponds to a solution of the special Lagrangian equation \eqref{SPL}.  In this section, we explore the geometry of the corresponding special Lagrangian submanifold.

\begin{prop} \label{prop_correspondence}
Let $p_1(s),\cdots,p_n(s)$ and $r(s)$ be solutions of the $\theta$-angle ODE system in Definition \ref{theta_ODE}. Consider the following function $f$ defined on a domain $X\subset \BR^{n+1}$ by
\begin{align*}
    f(x_1,\ldots, x_n, s)=\frac{1}{2}\sum_{j=1}^n p_j(s) x_j^2+r(s) ~.
\end{align*}
Then, $f$ satisfies the special Lagrangian equation with angle $\theta$:
\begin{align} \label{SPL}
    \im\left( e^{-i\ta}\det\left(\bfI_{n+1} + i\left[\frac{\partial^2 f}{\partial{x_j}\partial{x_k}}\right]_{1\leq j, k\leq n+1} \right)\right) &= 0 ~.
\end{align}
\end{prop}

\begin{proof}
Write $s$ as $x_{n+1}$.  The Hessian matrix of $f$ is
\begin{align*}
    \left[\frac{\partial^2 f}{\partial x_j \partial x_k} \right]_{1 \leq j,k \leq n+1} =
\begin{bmatrix}
p_1(s) & 0 & \cdots & 0 & x_1 p_1'(s) \\
0 & p_2(s) & \cdots & 0 & x_2 p_2'(s) \\
\vdots & \vdots & \ddots & \vdots & \vdots \\
0 & 0 & \cdots & p_n(s) & x_n p_n'(s) \\
x_1 p_1'(s) & x_2 p_2'(s) & \cdots & x_n p_n'(s) & \sum_{j=1}^n \frac{1}{2} p_j''(s) x_j^2 + r''(s)
\end{bmatrix} ~.
\end{align*}
With this, the computation is the same as that in Section~\ref{sec_LYZ}.
\end{proof}





The correspondence also holds true for the more general ansatz \eqref{general_potential}.  To be more precise, suppose that $p_j(s)$, $q_j(s)$ and $r(s)$ obey the system of equations \eqref{with_linear_01}, \eqref{with_linear_02} and \eqref{with_linear_03}.  Then, $f = \sum_{j=1}^n (\oh p_j(s)(x_j)^2 + q_j(s)x_j) + r(s)$ satisfies \eqref{SPL}.  Therefore, Proposition~\ref{prop_entire3} and Theorem~\ref{thm_entiren_sub} lead to the following theorem:

\begin{thm} \label{thm_slag_entire}
    \begin{itemize}
        \item For any $\ta\in[-\frac{\pi}{2},\frac{\pi}{2}]$, the special Lagrangian equation with angle $\theta$ \eqref{SPL} admits non-polynomial entire solutions on $\BR^{3}$.  The phase of the solution is exactly $\ta$.
        \item For \(n\geq 3\) and any \(|\ta| < (n-1)\frac{\pi}{2}\), the special Lagrangian equation with angle $\theta$ \eqref{SPL} admits non-polynomial entire solutions on $\BR^{n+1}$.  The phase of the solution is $\ta$.
    \end{itemize}
\end{thm}

\subsection{Extensions of solutions to the special Lagrangian equation}

According to Proposition~\ref{prop_correspondence}, the graph of $\nabla f$ defines a special Lagrangian submanifold in $\mathbb{C}^{n+1}$, which is graphical on a domain $X \subset \mathbb{R}^{n+1}$.  However, $\nabla f$ may blow up at the boundary of $X$, and the corresponding special Lagrangian submanifold ceases to be graphical.  In this subsection, we demonstrate that the submanifold admits a global extension (cf.\ \cites{Harvey-Lawson, Lawlor, Joyce}) as a complete (non-graphical), embedded, special Lagrangian in $\mathbb{C}^{n+1}$. It is natural to ask whether similar extension mechanisms can be applied to the dHYM/LYZ formulation on the mirror side.

We begin by recalling the following result of Joyce.

\begin{thm}[\cite{Joyce}*{Theorem\footnote{The dimension $m$ in Joyce's theorem corresponds to $n+1$ here.  For convenience, we specialize to the case $a = n$.  One can also work out the transformation for quadrics of other signature.} 7.1}] \label{thm_Joyce}
Let $w_1, \dots, w_n: (-\vep, \vep) \to \BC \setminus \{0\}$ and $\bt: (-\vep, \vep) \to \BC \setminus \{0\}$ be differentiable functions satisfying
\begin{align} \label{qu7eq3}
\begin{split}
\frac{\dd w_j}{\dd t} &= \overline{w_1 \cdots w_{j-1} w_{j+1} \cdots w_n}, \quad j = 1, \dots, n ~, \\
\frac{\dd\bt}{\dd t} &= \overline{w_1 \cdots w_n} ~.
\end{split}
\end{align}
Define a subset $N \subset \mathbb{C}^{n+1}$ by
\begin{align} \label{qu7eq4}
N = \left\{ \left(w_1(t)\xi_1, \dots, w_n(t)\xi_n, -\frac{\xi_1^2 + \dots + \xi_n^2}{2} + \beta(t) \right) : t \in (-\varepsilon, \varepsilon), \xi_j \in \mathbb{R} \right\}.
\end{align}
Then $N$ is a special Lagrangian submanifold of $\mathbb{C}^{n+1}$.
\end{thm}

To relate our construction with Joyce’s theorem, we begin with the expression for the graph of $\nabla f$:
\begin{align*}
(x_1, \ldots, x_n, s) \mapsto \left((1 + i p_1(s)) x_1, \ldots, (1 + i p_n(s)) x_n, s + i\left( \frac{1}{2} \sum_{j=1}^n x_j^2 p_j'(s) + r'(s) \right) \right)
\end{align*}
where $p_1(s),\ldots,p_n(s)$ and $r(s)$ satisfy \eqref{ODE_p} and \eqref{ODE_r}.  The angle $\ta$ will be specified later.  Moreover, assume that the constant \eqref{conserve0} is positive,
\begin{align*}
\frac{\prod_{j=1}^n p_j'(s)}{(F_\theta)^2} = \kp > 0 ~,~ \text{ and assume that }~ p_j'(s) > 0
\end{align*}
for $j=1,\ldots,n$.  Other cases can be treated similarly by appropriately adjusting signs, and they correspond to quadrics of other signatures.

We claim that our ansatz corresponds to the solution given by Theorem~\ref{thm_Joyce} through the following relation:
\begin{align}
\label{cor-eqn}
\begin{split}
&\left((1 + i p_1(s)) x_1, \ldots, (1 + i p_n(s)) x_n, s + i\left( \frac{1}{2} \sum_{j=1}^n x_j^2 p_j'(s) + r'(s) \right)\right) \\
&= -i \left(w_1(t) \xi_1, \ldots, w_n(t) \xi_n, -\frac{\xi_1^2 + \cdots + \xi_n^2}{2} + \beta(t)\right) ~.
\end{split}
\end{align}
To facilitate this, introduce the parameters $(\xi_1, \ldots, \xi_n, t)$ related to $(x_1, \ldots, x_n, s)$ by
\begin{align*}
t = \sqrt{\kp} s \quad\text{and}\quad \xi_j = \sqrt{p_j'(s)}\,x_j
\end{align*}
for $j=1\ldots,n$.
Define the complex-valued functions $\omega_j(t)$ and $\beta(t)$ by
\begin{align*}
\omega_j(t) &= i \left( \frac{1 + i p_j(s)}{\sqrt{p_j'(s)}} \right) ~, \\
\beta(t) &= i(s + i r'(s)) ~.
\end{align*}


It remains to verify that this parameterization satisfies the ODE system \ref{qu7eq3}.  By \eqref{ODE_p},
\begin{align*}
    \frac{\dd}{\dd s} \left( \frac{1 + i p_j(s)}{\sqrt{p_j'(s)}} \right) &= i \sqrt{p_j'} - (1 + i p_j) \sqrt{p_j'} \cdot \frac{\partial_{p_j} F_\theta}{F_\theta}.
\end{align*}
With the identity $F_\theta - i F_{\theta + \frac{\pi}{2}} = e^{i\theta} \bar{F}$ and the relation $p_j F_\theta - (1 + p_j^2) \partial_{p_j} F_\theta = F_{\theta + \frac{\pi}{2}}$ (see Lemma \ref{lem_Ft}), it can be simplified as
\begin{align*}
    \frac{\dd}{\dd s} \left( \frac{1 + i p_j}{\sqrt{p_j'}} \right) = \frac{\sqrt{p_j'}}{1 - i p_j} \cdot \frac{i e^{i\theta} \bar{F}}{F_\theta} ~,
\end{align*}
and hence
\begin{align*}
    \frac{\dd\om_j}{\dd t} = \frac{i}{\bar{\om}_j} \cdot \overline{ \left( \frac{e^{-i\ta} F}{F_\ta} \right) } \cdot \frac{\dd s}{\dd t} ~.
\end{align*}
Similarly, it follows from \eqref{ODE_r} that
\begin{align*}
    \frac{\dd\bt}{\dd t} = i (1 + i r''(s)) \frac{\dd s}{\dd t} = i \overline{ \left( \frac{e^{-i\ta} F}{F_\ta} \right) } \cdot \frac{\dd s}{\dd t} ~.
\end{align*}

On the other hand,
\begin{align*}
    \omega_1 \cdots \omega_n = \prod_{j=1}^n \left( \frac{i(1 + i p_j(s))}{\sqrt{p_j'(s)}} \right) = \frac{i^n F}{\sqrt{\kp} F_\ta} ~.
\end{align*}
Therefore, $\omega_j(t)$ and $\bt(t)$ satisfy
\begin{align*}
    \frac{\dd \om_j}{\dd t} = i \overline{ \left( \frac{e^{-i\ta} \om_1 \cdots \om_n}{i^n \om_j} \right) }
    \quad\text{ and }\quad
    \frac{\dd\bt}{\dd t} = i \overline{ \left( \frac{e^{-i\ta} \om_1 \cdots \om_n}{i^n} \right) } ~.
\end{align*}
Finally, by choosing the angle $\ta$ such that $e^{i\ta} i^n = -i$, we recover the ODE system \eqref{qu7eq3}.  This verifies the correspondence \ref{cor-eqn}.

\section{General equations of recursive type} \label{sec_general}

In this section, we fix real constants $c_n,\ldots,c_{-1}$ and study solutions of a general real Hessian equation of the form
\begin{align}
    c_n\sm_{n+1}(\nabla^2 f) + c_{n-1}\sm_n(\nabla^2 f) + \cdots + c_0\sm_1(\nabla^2 f) + c_{-1} &= 0.
\label{rH_eqn} \end{align}
Denote the coordinate of $\BR^{n+1}$ by $(x_1,\ldots,x_n,s)$, and consider the following ansatz for $f$:
\begin{align}
    f &= \frac{1}{2}\sum_{j=1}^n p_j(s) x_j^2+ r(s) ~.
\label{ansatz_f1} \end{align}

We also consider solutions of a general complex Hessian equation of the form
\begin{align}
    c_n\sm_{n+1}(\partial\bar{\partial} u) + c_{n-1}\sm_n(\partial\bar{\partial} u) + \cdots + c_0\sm_1(\partial\bar{\partial} u) + c_{-1} &= 0.
\label{cH_eqn} \end{align}
Denote the coordinate of $\mathbb{C}^{n+1}$ by $(z_1,\ldots,z_n,z_{n+1})$, and consider the following ansatz for $u$:
\begin{align}
    u &= 2\sum_{j=1}^n p_j(\re z_{n+1}) (\re z_i)^2 + 4r(\re z^{n+1}) ~.
\label{ansatz_u1} \end{align}

\begin{prop}\label{general_ODE} 
Suppose $p_i(s), i=1\ldots n$ and $r(s)$ satisfy the following ODE system   
\begin{align}
    \frac{p_i''}{2}\cdot F(p_1,\ldots,p_n)  &= (p_i')^2 \frac{\pl F}{\pl p_i}(p_1,\ldots,p_n)  \quad\text{for $i=1,\ldots,n$ and} \label{ODE_pi_g} \\
    r''\cdot F(p_1,\ldots,p_n)  &= -G(p_1,\ldots,p_n), \label{ODE_r_g} 
\end{align} where
\begin{align*}
    F(p_1,\ldots,p_n) = \sum_{k=0}^n c_k\sm_k(p) \quad\text{and}\quad G(p_1,\ldots,p_n) = \sum_{k=0}^n c_{k-1}\sm_k(p) ~. 
\end{align*}
Then a function $u$ of the form \eqref{ansatz_u1} satisfies the complex Hessian equation \eqref{cH_eqn}, and a function $f$ of the form \eqref{ansatz_f1} satisfies the real Hessian equation \eqref{rH_eqn}.

\end{prop}

\begin{proof} The proof extends the argument of Proposition \ref{prop_basic} (the dHYM/LYZ case) and Proposition \ref{prop_correspondence} (the special Lagrangian case). We only deal with the real case \eqref{rH_eqn} here and the complex case \eqref{cH_eqn} can be dealt with similarly.
    By Lemma~\ref{lem_det}, \eqref{rH_eqn}, under the ansatz \eqref{ansatz_f1}, becomes
\begin{align}
    0 &= \Xi_0(s) + \sum_{i=1}^n(x^i)^2 \Xi_i(s)
\label{realH_eqn1} \end{align} where
\begin{align*}
    \Xi_0(s) &= \left( c_{n-1}\sm_n(p) + \cdots + c_{0}\sm_1(p) + c_{-1} \right) + \left( c_n\sm_n(p) + \cdots + c_1\sm_1(p) + c_0 \right) r'' \quad\text{and}  \\
    \Xi_i(s) &= \left( c_n\sm_n(p) + \cdots + c_1\sm_1(p) + c_0 \right) \frac{p_i''}{2} - \left( c_n\sm_{n-1}(p|i) + \cdots + c_1\sm_1(p|i) + c_1 \right) (p_i')^2  ~.
\end{align*}
As before, $\sm_k(p|i) = \sm_k(p_1,\ldots,\widehat{p_i},\ldots,p_n)$.  Note that \eqref{realH_eqn1} is equivalent to $\Xi_i(s) = 0 = \Xi_0(s)$ for $i=1,\ldots,n$.
\end{proof}

Again, the main task is to solve \eqref{ODE_pi_g} for $p_i$, and then $r$ can be found by integrating \eqref{ODE_r_g} twice.  Let $R_i(s)$ be $1/p_i'(s)$, and \eqref{ODE_pi_g} becomes the following first order system

\begin{equation}\begin{cases} 
    p_i' &= \frac{1}{R_i} ~,\\
    R_i' &= -2\frac{\pl \log F}{\pl p_i}
 \label{ODE_pR} \end{cases}\end{equation}
for $i=1,\ldots,n$.

\begin{rmk} \label{rmk_first_integral}
It is not hard to see that with the symplectic form $\sum_{i=1}^n \dd p_i \w \dd R_i$, the ODE system is Hamiltonian with respect to $H(p_i, R_i)=\sum_{i=1}^n \log R_i +2\log F$ and thus $R_1\cdots R_n F^2$ is a first integral. However, since there is no other continuous symmetry of $H$ for general $F$, this perspective is not particularly useful.  
\end{rmk}

\begin{defn}
    The ODE system \eqref{ODE_pi_g} is said to be of recursive type $(a_0, a_1)$ if there exist real numbers $a_0$ and $a_1$ such that the coefficients of $F$ satisfy the recursive relation:
    \begin{align*}
        c_{k-1} = c_k a_1 - c_{k+1} a_0 \quad \text{for } k=1, \dots, n-1 ~.
    \end{align*}
\end{defn}

In particular, $c_0, c_1, \cdots, c_{n-2}$ are determined by $a_0, a_1, c_{n-1}$ and $c_n$.
All recursive types $F$ can be classified according to the following proposition ($c_n$ and $c_{n-1}$ are not necessarily real in this proposition): 

\begin{prop}\label{classification}
Let $n \geq 1$ and let $a_0, a_1 \in\BR$.  Suppose
\begin{align*}
    F(p_1,\dots,p_n) &= \sum_{k=0}^n c_k\,\sigma_k(p_1,\dots,p_n)
\end{align*}
is a symmetric polynomial in $(p_1,\dots,p_n)$, where $\sm_k$ denotes the $k$-th elementary symmetric function. Assume the coefficients $\{c_k\}$ satisfy the recurrence
\begin{align*}
    c_{k-1} &= a_1\,c_k - a_0\,c_{k+1}~, \qquad k = 1,2,\dots,n-1 ~.
\end{align*}

Let \(r_1, r_2\) be the (not necessarily distinct) roots of the quadratic equation
\begin{align*}
    r^2 - a_1 r + a_0 &= 0 ~.
\end{align*}
Then \(F\) must take one of the following forms:

\underline{Case 1}: \(r_1 \neq r_2\).
\begin{align*}
    F(p_1,\dots,p_n) &= A \prod_{i=1}^n (p_i + r_1) + B \prod_{i=1}^n (p_i + r_2) ~,
\end{align*}
where the constants \(A, B\) are given by
\begin{align*}
    A = \frac{c_{n-1} - c_n r_2}{r_1 - r_2} \quad\text{and}\quad B = \frac{-(c_{n-1}-c_n r_1)}{r_1 - r_2} ~.
\end{align*}

\underline{Case 2}: \(r_1 = r_2 = u \neq 0\).
\begin{align*}
    F(p_1,\dots,p_n) &= A \prod_{i=1}^n (p_i + u) + B u \cdot \frac{\dd}{\dd u} \Bigl( \prod_{i=1}^n (p_i + u) \Bigr) ~,
\end{align*}
where \(A = c_n\) and \(B = \frac{c_{n-1}}{u} - c_n\).

\underline{Case 3}: \(r_1 = r_2 = 0\).
\begin{align*}
    F(p_1,\dots,p_n) &= c_n\,\sigma_n(p_1,\dots,p_n) +c_{n-1}\,\sigma_{n-1}(p_1,\dots,p_n) ~.
\end{align*}
\end{prop}

\begin{proof}
\underline{Case 1}: \(r_1 \neq r_2\).
One can verify directly that the sequence \(B_k = A r_1^{n-k} + B r_2^{n-k}\) satisfies the recurrence
$c_{k-1} = a_1\,c_k - a_0\,c_{k+1}$ for $k = 1,2,\dots,n-1$, with initial conditions \(B_n = c_n\), \(B_{n-1} = c_{n-1}\).  Moreover, the identity
\begin{align*}
    \sum_{k=0}^n B_k\,\sigma_k(p_1,\dots,p_n) &= A \prod_{i=1}^n (p_i + r_1) + B \prod_{i=1}^n (p_i + r_2)
\end{align*}
follows from the generating function for elementary symmetric polynomials.

\underline{Case 2}: \(r_1=r_1=u\neq0\).
In this case, the recurrence becomes \(c_{k-1} = 2u\,c_k - u^2\,c_{k+1}\). It is not hard to verify that
$B_k = A u^{n-k} + B(n - k) u^{n-k}$ satisfies the recurrence, with initial conditions \(B_n = c_n\), \(B_{n-1} = c_{n-1}\). Furthermore,
\begin{align*}
    \sum_{k=0}^n B_k\,\sigma_k(p_1,\dots,p_n) &= A \prod_{i=1}^n (p_i + u) + B u \cdot \frac{\dd}{\dd u} \Bigl( \prod_{i=1}^n (p_i + u) \Bigr)
\end{align*}
follows from term-wise differentiation of the generating polynomial.

\underline{Case 3}: $r_1=r_2=0$.
In this case, the recurrence becomes \(c_{k-1} = 0\), implying \(c_0 = \cdots = c_{n-2} = 0\). Only \(c_{n-1}\) and \(c_n\) may be nonzero, and thus $F(p) = c_n \sigma_n(p)+c_{n-1} \sigma_{n-1}(p)$.
\end{proof}

In case 1, by setting $a_0=1, a_1=0$, we have $r_1=i, r_2=-i$ , and this corresponds to the dHYM/LYZ equation in the complex case and the special Lagrangian equation in the real case. 
Case 3 gives the Monge--Amp\`ere equation and the $J$-equation.

\begin{thm}
   Suppose the ODE system \eqref{ODE_pi_g} is of recursive type $(a_0, a_1)$. Define
\begin{align*}
    \xi_i = \frac{p_i^2 + a_1 p_i + a_0}{p'_i} ~, \quad i=1, \ldots, n~, 
\end{align*}
then
\begin{align*}
    \xi_i - \xi_1 ~, \quad i=2, \ldots, n
\end{align*}
are first integrals of the system.  In addition, $\xi_i$, $i=1,\ldots,n$ satisfy the following ODE system: \begin{align*}
 (\xi_i)'^2=(a_1^2-4a_0)+\frac{4}{\kp}(c_{n-1}^2-a_1 c_{n-1}c_n+a_0 c_n^2) \prod_{j=1}^n\xi_j~,
\end{align*}
where $\kp$ is the constant such that $F^2=\kp\,p_1'\cdots p_n'$.
\end{thm}

\begin{proof}
Let $\xi_i= R_i (p_i^2+a_1p_i+a_0)$. It follows from a direct computation that
\begin{align*}
 \xi_i'= \frac{1}{F}  (-2 (p_i^2+a_1p_i+a_0)\frac{\partial F}{\partial p_i} + (2p_i+a_1)F).
\end{align*}
By Lemma \ref{quadratic_lemma},
\begin{align*}
    (\xi'_i)^2=(a_1^2-4a_0)+\frac{4}{F^2}(c_{n-1}^2-a_1 c_{n-1}c_n+a_0 c_n^2) \prod_{j=1}^n(p_j^2+a_1p_j+a_0)
\end{align*}
This together with $F^2=c\,p_1'\cdots p_n'$ and the definition of $\xi_i$ finishes the proof of this theorem.
\end{proof}

\subsection{Complete classification when $n=3$} \label{sec_3dnr}

In this subsection, we consider \eqref{ODE_pi_g} when $n=3$.  In this case, $F(p_1,p_2,p_3) = c_3\sm_3(p) + c_2\sm_2(p) + c_1\sm_1(p) + c_0$.  When $c_3c_1 - (c_2)^2 \neq 0$, it is easy to see that \eqref{ODE_pi_g} is of recursive type, and $(a_0,a_1)$ are unique.  It remains to analyze \eqref{ODE_pi_g} when $c_3c_1 = (c_2)^2$.  With Remark~\ref{rmk_first_integral}, $R_1R_2R_3\,F^2 = c$ for some constant $c$.  It suffices to identify two more independent first integrals.

\begin{prop} \label{prop_first_integral3}
    When $n=3$ and $c_3c_1 = (c_2)^2$, the ODE system \eqref{ODE_pi_g} admits the following first integrals.
    \begin{enumerate}
        \item If $c_3\neq0$, let $a = c_2/c_3$.  Then, $(p_1+a)R_1 - (p_j+a)R_j, j=2,3$ are first integrals.
        \item If $c_3 = 0$, $R_1 - R_j, j=2,3$ are first integrals.
    \end{enumerate}
\end{prop}
\begin{proof}
    \underline{Case (i)}.  In this case,
    \begin{align}
        F(p_1,p_2,p_3) &= c_3(p_1+a)(p_2+a)(p_3+a) + (c_0 - c_3a^3) ~. \label{F_nonrecur1}
    \end{align}
    With this, it follows from a direct computation that the derivative of $(p_j+a)R_j$ is independent of $j$.

    \underline{Case (ii)}.  In this case,
    \begin{align}
        F(p_1,p_2,p_3) &= c_1(p_1+p_2+p_3) + c_0 ~. \label{F_nonrecur2}
    \end{align}
    It is clear that the derivative of $R_j$ is independent of $j$.
\end{proof}

We now explain how to use Proposition~\ref{prop_first_integral3} to solve \eqref{ODE_pR}.  In case (i), denote $(p_j+a)R_j$ by $\xi_j$ for $j=1,2,3$.  Let $\kp_j$ be the constants $\xi_1-\xi_j$ for $j=2,3$.  The equation $R_1R_2R_3\,F^2 = c$ becomes
\begin{align}
    \xi_1(\xi_1 - \kp_2)(\xi_1 - \kp_3) = \frac{c(p_1+a)(p_2+a)(p_3+a)}{(F(p_1,p_2,p_3))^2} ~. \label{cubic_nonrecur1}
\end{align}

In other words, $\xi_1 = \xi_1(p_1,p_2,p_3)$ is the root of the cubic polynomial \eqref{cubic_nonrecur1} that is compatible with the initial condition of the system \eqref{ODE_pR}.  Therefore, \eqref{ODE_pR} reduces to the following autonomous system:
\begin{align*}
    p_1' &= \frac{p_1+a}{\xi_1(p_1,p_2,p_3)} ~,
    & p_2' &= \frac{p_2+a}{\xi_1(p_1,p_2,p_3) - \kp_2} ~,
    & p_3' &= \frac{p_3+a}{\xi_1(p_1,p_2,p_3) - \kp_3} ~.
\end{align*}

For care (ii), let $\kp_j$ be the constants $R_1-R_j$ for $j=2,3$.  A similar computation gives
\begin{align}
    R_1(R_1 - \kp_2)(R_1 - \kp_3) = \frac{c}{(F(p_1,p_2,p_3))^2} ~. \label{cubic_nonrecur2}
\end{align}
It follows that $R_1 = R_1(p_1,p_2,p_3)$ is the root of the cubic polynomial \eqref{cubic_nonrecur2} that is compatible with the initial condition of the system \eqref{ODE_pR}, and \eqref{ODE_pR} reduces to the following autonomous system:
\begin{align*}
    p_1' &= \frac{1}{R_1(p_1,p_2,p_3)} ~,
    & p_2' &= \frac{1}{R_1(p_1,p_2,p_3) - \kp_2} ~,
    & p_3' &= \frac{1}{R_1(p_1,p_2,p_3) - \kp_3} ~.
\end{align*}

\appendix
\section{Some Algebraic Calculations}

\begin{lem} \label{lem_det}
For the $(n+1)\times (n+1)$ Hermitian matrix
\begin{align*}
H_{n+1} &= \begin{bmatrix}
P_1 & 0 & \cdots & 0 & Q_1 \\  0 & P_2 & \cdots & 0 & Q_2 \\  \vdots & \vdots & \ddots & \vdots & \vdots \\
0 & 0 & \cdots & P_{n} & Q_{n} \\  \bar{Q}_1 & \bar{Q}_2 & \cdots & \bar{Q}_{n} & R
\end{bmatrix} ~,
\end{align*}
\begin{align*}
\det(\ld \bfI_{n+1} - H_{n+1}) &= (\ld-P_1)(\ld-P_2)\cdots(\ld-P_n)(\ld-R) \\
&\quad - \sum_{i=1}^{n} |Q_i|^2 (\ld-P_1)\cdots\widehat{(\ld-P_i)}\cdots(\ld-P_{n}) ~.
\end{align*}
\end{lem}

\begin{proof}
When $n=2$ or $3$, the assertion can be proved by a direct computation.

Suppose this lemma is true when the size is no greater than $n$.  When the size is $n+1$, expand $\det(\ld \bfI_{n+1} - H_{n+1})$ along the first column.
\begin{align*}
\det\begin{bmatrix}
\ld-P_1 & 0 & \cdots & 0 & -Q_1 \\  0 & \ld-P_2 & \cdots & 0 & -Q_2 \\  \vdots & \vdots & \ddots & \vdots & \vdots \\
0 & 0 & \cdots & \ld-P_{n} & -Q_{n} \\  -\bar{Q}_1 & -\bar{Q}_2 & \cdots & -\bar{Q}_{n} & \ld-R  \end{bmatrix}
&= (\ld-P_1)\cdot\det\begin{bmatrix}
\ld-P_2 & \cdots & 0 & -Q_2 \\  0 & \ddots & \vdots & \vdots \\
0 & \cdots & \ld-P_{n} & -Q_{n} \\  -\bar{Q}_2 & \cdots & -\bar{Q}_{n} & \ld-R  \end{bmatrix} \\
&\quad + (-1)^{n}(-\bar{Q}_1)\cdot\det\begin{bmatrix}
0 & \cdots & 0 & -Q_1 \\  \ld-P_2 & \cdots & 0 & -Q_2 \\  \vdots & \ddots & \vdots & \vdots \\
0 & \cdots & \ld-P_{n} & -Q_{n}  \end{bmatrix}
\end{align*}
It follows from the induction hypothesis that the first term on the right hand side is equal to
\begin{align*}
(\ld - P_1)\cdot\left[ (\ld-P_2)\cdots(\ld-P_{n})(\ld- R) - \sum_{j=2}^{n}|Q_j|^2(\ld-P_2)\cdots\widehat{(\ld-P_j)}\cdots(\ld-P_{n}) \right] ~.
\end{align*}
A direct computation on the determinant shows that the second term on the right hand side is equal to
\begin{align*}
(-1)^{n}(-\bar{Q}_1)\cdot\left[ (-1)^{n-1}(-Q_1)(\ld-P_2)\cdots(\ld-P_{n}) \right] ~.
\end{align*}
Putting these together finishes the proof of this lemma.
\end{proof}

\begin{lem}\label{quadratic_lemma}
Let $a_0$  $a_1$, $c_{n-1}$, and $c_n$ be  real numbers. Let $q(x)$ be the quadratic polynomial $q(x) = x^2 + a_1 x + a_0$ and  $F(p_1, \dots, p_n)$ be the symmetric polynomial in $p_1, \dots, p_n$ given by
\begin{align*}
    F(p_1, \dots, p_n) = \sum_{k=0}^n c_k \sigma_k
\end{align*}
with \[c_{k-1} = a_1 c_{k} - a_0 c_{k+1}, k = 1, \dots, n-1\] and $\sigma_k$ the $k$-th symmetric function in $p_1, \cdots, p_n$. Then for every $i = 1, \ldots, n$
\begin{align}\label{structure_id} (q'(p_i) F - 2 q(p_i) \partial_{p_i} F)^2 = (a_1^2-4a_0)F^2+4(c_{n-1}^2-a_1 c_{n-1}c_n+a_0 c_n^2) \prod_{j=1}^n q(p_j) ~.
\end{align}
\end{lem}

\begin{proof}
Note that $F$ only depends on the coefficients $a_0, \cdots, a_n$.
For the sake of the proof, we introduce a temporary constant
\begin{align*}
    c_{-1}=a_1c_0-a_0c_1
\end{align*}
which is distinct from the earlier $c_{-1}$ and  does not appear in the final formula. 

The first step is to prove that 
\begin{align}\label{structure_id_0}
q'(p_i) F-2q(p_i)\partial_{p_i} F=-a_1F+2\sum_{k=0}^n c_{k-1} \sigma_k ~.
\end{align}

 For $i=1,\ldots, n$ and $k=0,\ldots, n-1$, denote by $\sigma_k(p|i)$ the $k$-th symmetric function of $p_1, \cdots, p_n$ with $p_i$ excluded. We have 
\begin{align}\label{sigma} \sigma_k(p) = p_i \sigma_{k-1}(p|i) + \sigma_k(p|i) ~, \quad k=0,\ldots,n \end{align}
where we adopt the convention that \(\sigma_{-1}(p|i) = 0\) and \(\sigma_n(p|i) = 0\).
In particular, \(\partial_{p_i} F = \sum_{k=0}^{n} c_k \sigma_{k-1}(p|i)\).

With these, we compute that $q'(p_i) F-2q(p_i)\partial_{p_i} F$ is given by
\begin{align*} &\quad (2p_i + a_1) \sum_{k=0}^{n} c_k (p_i \sigma_{k-1}(p|i) + \sigma_k(p|i)) - 2(p_i^2 + a_1 p_i + a_0) \sum_{k=0}^{n} c_k \sigma_{k-1}(p|i)\\
&= (-a_1p_i-2a_0)\sum_{k=0}^{n} c_k  \sigma_{k-1}(p|i) + (2p_i + a_1) \sum_{k=0}^{n} c_k \sigma_k(p|i)\\
 &= (-a_1 p_i - 2a_0)\sum_{k=0}^{n-1} c_{k+1}  \sigma_k(p|i) + (2p_i + a_1)\sum_{k=0}^{n-1} c_k  \sigma_k(p|i) \\
& = \sum_{k=0}^{n-1} \Bigl[ p_i (2 c_{k} - a_1 c_{k+1})+c_k a_1 - 2 c_{k+1} a_0 \Bigr] \sigma_k(p|i)\\
& = \sum_{k=0}^{n} p_i (2 c_{k-1} - a_1 c_{k})\sigma_{k-1}(p|i)+\sum_{k=0}^{n-1}(c_k a_1 - 2 c_{k+1} a_0)\sigma_k(p|i), \end{align*}
where the indexes are shifted after the second and the fourth equalities.
With the recursive relation and the definition of $c_{-1}$, we have $c_k a_1-2c_{k+1}a_0=2c_{k-1}-a_1 c_k$ for $k=0,\cdots n-1$. Regrouping terms yields:
\begin{align*}
    \sum_{k=0}^{n} (2 c_{k-1} - a_1 c_{k})(p_i\sigma_{k-1}(p|i)+ \sigma_k(p|i)) ~.
\end{align*}
Applying \eqref{sigma}, we obtain the desired expression
\begin{align*}
    -a_1F+2\sum_{k=0}^n c_{k-1} \sigma_k(p)
\end{align*}
and complete the proof of \eqref{structure_id_0}.

We now verify the identity \eqref{structure_id} case-by-case by using Proposition~\ref{classification}.

\underline{Case 1}: \( r_1 \neq r_2 \).
Suppose \( r_1 \ne r_2 \) are distinct (real or complex conjugate) roots of the characteristic equation $r^2 - a_1 r + a_0 = 0$.  Then, \( a_1 = r_1 + r_2 \) and \( a_0 = r_1 r_2 \). By Proposition~\ref{classification}, the function \(F\) must be of the form
\begin{align}\label{case1_F}
    F = \sum_{k=0}^n (A r_1^{n-k} + B r_2^{n-k}) \sigma_k = A\,P + B\,Q,
\end{align}
where
\begin{align*}
    P = \sum_{k=0}^n r_1^{n-k} \sigma_k = \prod_{j=1}^n (p_j + r_1) \quad\text{and}\quad
    Q = \sum_{k=0}^n r_2^{n-k} \sigma_k = \prod_{j=1}^n (p_j + r_2)
\end{align*}
for some constants \( A, B \). (These constants may be complex if \( r_1 \) and \( r_2 \) are complex conjugates.)

The coefficients of \(F\) are therefore given by $c_k = A r_1^{n-k} + B r_2^{n-k}$.
Using this, we compute
\begin{align}\label{case1_other}
    \sum_{k=0}^n c_{k-1} \sigma_k = A r_1 P + B r_2 Q ~.
\end{align}
Substituting \eqref{case1_F} and \eqref{case1_other}, we find
\begin{align*}
    \biggl( -a_1 F + 2 \sum_{k=0}^n c_{k-1} \sigma_k \biggr)^2 &= (r_1 - r_2)^2 \biggl[ F^2 - 4AB \prod_{j=1}^n q(p_j) \biggr] ~,
\end{align*}
where \( q(p_j) = (p_j + r_1)(p_j + r_2) \).
Expressing \((r_1 - r_2)^2\) and \(AB\) in terms of \(a_0\), \(a_1\), \(c_{n-1}\), and \(c_n\), we obtain
\begin{align*}
    \biggl( -a_1 F + 2 \sum_{k=0}^n c_{k-1} \sigma_k \biggr)^2 &= (a_1^2 - 4a_0) F^2 + 4(c_{n-1}^2 - a_1 c_{n-1} c_n + a_0 c_n^2) \prod_{j=1}^n q(p_j) ~.
\end{align*}

\underline{Case 2}: \( r_1 = r_2 = u \neq 0 \).  Suppose the characteristic equation has a repeated root \( u \), so that \( a_1 = 2u \) and \( a_0 = u^2 \). Then \(F\) takes the form
\begin{align}\label{case2_F}
    F &= \sum_{k=0}^n \left( A + B(n-k) \right) u^{n-k} \sigma_k = A\,P + B\,Q ~,
\end{align}
where
\begin{align*}
    P = \sum_{k=0}^n u^{n-k} \sigma_k = \prod_{j=1}^n (p_j + u) \quad\text{and}\quad
    Q = \sum_{k=0}^n (n - k) u^{n-k} \sigma_k = u \frac{\dd P}{\dd u} ~.
\end{align*}

Thus, the coefficients of \(F\) are
\begin{align} \label{C_k_repeated}
    c_k = A u^{n-k} + B(n-k) u^{n-k} ~.
\end{align}
From this, we compute
\begin{align}\begin{split}\label{case2_other}
\sum_{k=0}^n c_{k-1} \sigma_k
&= \sum_{k=0}^n \Bigl[ A u^{n - k + 1} + B(n - k + 1) u^{n - k + 1} \Bigr] \sigma_k \\
&= u \sum_{k=0}^n \Bigl[ A u^{n - k} + B(n - k) u^{n - k} + B u^{n - k} \Bigr] \sigma_k \\
&= u (A + B) P + u B Q ~.
\end{split}\end{align}

Combining \eqref{case2_F} and \eqref{case2_other}, we have
\begin{align*}
    - a_1 F + 2 \sum_{k=0}^n c_{k-1} \sigma_k &= a_1 B P ~,
\end{align*}
and hence
\begin{align*}
    \Bigl( -a_1 F + 2 \sum_{k=0}^n c_{k-1} \sigma_k \Bigr)^2 = a_1^2 B^2 \prod_{j=1}^n q(p_j) ~.
\end{align*}
Since \( a_1 = 2u \),
\begin{align*}
    B &= -\frac{2}{a_1} \Bigl( c_{n-1} - \frac{a_1}{2} c_n \Bigr) ~,
\end{align*}
so the identity becomes
\begin{align*}
    \Bigl( -a_1 F + 2 \sum_{k=0}^n c_{k-1} \sigma_k \Bigr)^2 &= 4 \Bigl( c_{n-1} - \frac{a_1}{2} c_n \Bigr)^2 \prod_{j=1}^n q(p_j) ~.
\end{align*}

\underline{Case 3}: \( r_1 = r_2 = 0 \).  In this case, \( a_0 = a_1 = 0 \), and the recurrence implies \( c_k = 0 \) for all \( k < n-1 \). Therefore, $F =c_n \sigma_n+ c_{n-1} \sigma_{n-1}$, and the identity \eqref{structure_id} follows immediately by a direct substitution.
\end{proof}



\begin{bibdiv}
\begin{biblist}

\bib{CNS85}{article}{
   author={Caffarelli, L.},
   author={Nirenberg, L.},
   author={Spruck, J.},
   title={The Dirichlet problem for nonlinear second-order elliptic equations. III. Functions of the eigenvalues of the Hessian},
   journal={Acta Math.},
   volume={155},
   date={1985},
   number={3-4},
   pages={261--301},
}

\bib{CLT24}{article}{
   author={Chu, Jianchun},
   author={Lee, Man-Chun},
   author={Takahashi, Ryosuke},
   title={A Nakai-Moishezon type criterion for supercritical deformed Hermitian-Yang-Mills equation},
   journal={J. Differential Geom.},
   volume={126},
   date={2024},
   number={2},
   pages={583--632},
}

\bib{CJY20}{article}{
   author={Collins, Tristan C.},
   author={Jacob, Adam},
   author={Yau, Shing-Tung},
   title={$(1,1)$ forms with specified Lagrangian phase: a priori estimates and algebraic obstructions},
   journal={Camb. J. Math.},
   volume={8},
   date={2020},
   number={2},
   pages={407--452},
}

\bib{CS22}{article}{
   author={Collins, Tristan C.},
   author={Shi, Yun},
   title={Stability and the deformed Hermitian-Yang-Mills equation},
   book={
      series={Surv. Differ. Geom.},
      volume={24},
      publisher={Int. Press, Boston, MA},
   },
   date={2022},
   pages={1--38},
}

\bib{Collins-Xie-Yau}{article}{
   author={Collins, Tristan C.},
   author={Xie, Dan},
   author={Yau, Shing-Tung},
   title={The deformed Hermitian-Yang-Mills equation in geometry and
   physics},
   conference={
      title={Geometry and physics. Vol. I},
   },
   book={
      publisher={Oxford Univ. Press, Oxford},
   },
   date={2018},
   pages={69--90},
}

\bib{Evans82}{article}{
   author={Evans, Lawrence C.},
   title={Classical solutions of fully nonlinear, convex, second-order elliptic equations},
   journal={Comm. Pure Appl. Math.},
   volume={35},
   date={1982},
   number={3},
   pages={333--363},
}

\bib{Harvey-Lawson}{article}{
   author={Harvey, Reese},
   author={Lawson, H. Blaine, Jr.},
   title={Calibrated geometries},
   journal={Acta Math.},
   volume={148},
   date={1982},
   pages={47--157},
}

\bib{Jacob-Yau}{article}{
   author={Jacob, Adam},
   author={Yau, Shing-Tung},
   title={A special Lagrangian type equation for holomorphic line bundles},
   journal={Math. Ann.},
   volume={369},
   date={2017},
   number={1-2},
   pages={869--898},
}

\bib{Joyce}{article}{
   author={Joyce, Dominic},
   title={Constructing special Lagrangian $m$-folds in $\BC^m$ by evolving quadrics},
   journal={Math. Ann.},
   volume={320},
   date={2001},
   number={4},
   pages={757--797},
}

\bib{Krylov82}{article}{
   author={Krylov, N. V.},
   title={Boundedly inhomogeneous elliptic and parabolic equations},
   language={Russian},
   journal={Izv. Akad. Nauk SSSR Ser. Mat.},
   volume={46},
   date={1982},
   number={3},
   pages={487--523, 670},
}

\bib{Lawlor}{article}{
   author={Lawlor, Gary},
   title={The angle criterion},
   journal={Invent. Math.},
   volume={95},
   date={1989},
   number={2},
   pages={437--446},
}

\bib{Leung-Yau-Zaslow}{article}{
   author={Leung, Naichung Conan},
   author={Yau, Shing-Tung},
   author={Zaslow, Eric},
   title={From special Lagrangian to Hermitian-Yang-Mills via Fourier-Mukai
   transform},
   journal={Adv. Theor. Math. Phys.},
   volume={4},
   date={2000},
   number={6},
   pages={1319--1341},
}

\bib{Li21}{article}{
   author={Li, Caiyan},
   title={Non-polynomial entire solutions to Hessian equations},
   journal={Calc. Var. Partial Differential Equations},
   volume={60},
   date={2021},
   number={4},
   pages={Paper No. 123, 6},
}

\bib{Lin23}{article}{
   author={Lin, Chao-Ming},
   title={The deformed Hermitian-Yang-Mills equation, the Positivstellensatz, and the solvability},
   journal={Adv. Math.},
   volume={433},
   date={2023},
   pages={Paper No. 109312, 71},
}

\bib{Lin23a}{article}{
   author={Lin, Chao-Ming},
   title={On the solvability of general inverse $\sm_k$ equations},
   journal={},
   volume={},
   date={},
   number={},
   pages={},
   eprint={arXiv:2310.05339},
   status={preprint},
}

\bib{WY14}{article}{
   author={Wang, Dake},
   author={Yuan, Yu},
   title={Hessian estimates for special Lagrangian equations with critical and supercritical phases in general dimensions},
   journal={Amer. J. Math.},
   volume={136},
   date={2014},
   number={2},
   pages={481--499},
}

\bib{Warren16a}{article}{
   author={Warren, Micah},
   title={Nonpolynomial entire solutions to $\sigma_k$ equations},
   journal={Comm. Partial Differential Equations},
   volume={41},
   date={2016},
   number={5},
   pages={848--853},
}

\bib{Warren16b}{article}{
   author={Warren, Micah},
   title={A Bernstein result and counterexample for entire solutions to Donaldson's equation},
   journal={Proc. Amer. Math. Soc.},
   volume={144},
   date={2016},
   number={7},
   pages={2953--2958},
}


\bib{Yuan06}{article}{
   author={Yuan, Yu},
   title={Global solutions to special Lagrangian equations},
   journal={Proc. Amer. Math. Soc.},
   volume={134},
   date={2006},
   number={5},
   pages={1355--1358},
}

\bib{Yuan20}{article}{
   author={Yuan, Yu},
   title={Special Lagrangian equations},
   book={
      series={Progr. Math.},
      volume={333},
      publisher={Birkh\"auser/Springer, Cham},
   },
   date={2020},
   pages={521--536},
}



\end{biblist}
\end{bibdiv}
\end{document}